\documentclass[11pt]{article}
\usepackage{mathrsfs}
\usepackage{amsmath, mathtools}
\usepackage{amssymb}
\usepackage{amsthm}
\usepackage{graphicx}
\usepackage{epic}
\usepackage{xcolor,cite}
\usepackage{cases}
\usepackage{pst-poly}
\usepackage{pst-plot}
\usepackage{pst-poly}
\usepackage{CJK}

\usepackage{verbatim}
\usepackage{subfig}

\definecolor{dkgreen}{rgb}{0,0.6,0}
\numberwithin{equation}{section}

\usepackage{etex}
\usepackage{pgf}

\renewcommand{\paragraph}{\roman{paragraph}}

\usepackage{bm}
\usepackage[hidelinks]{hyperref}
\usepackage{tikz}
\usetikzlibrary{automata}
\usepackage{enumitem}

\usetikzlibrary{arrows,shapes,positioning}
\usetikzlibrary{decorations.markings}
\tikzstyle arrowstyle=[scale=1]
\tikzstyle directed=[postaction={decorate,decoration={markings, mark=at position .65 with {\arrow[arrowstyle]{stealth}}}}]
\tikzstyle reverse directed=[postaction={decorate,decoration={markings, mark=at position .65 with {\arrowreversed[arrowstyle]{stealth};}}}]

\newtheorem{claim}{Claim}[section]
\newtheorem{theorem}{Theorem}[section]
\newtheorem{corollary}[theorem]{Corollary}
\newtheorem{definition}[theorem]{Definition}

\newtheorem{lemma}[theorem]{Lemma}

\begin{document}

\title{Maximum number of spanning trees and connectivity:\\ Graphs with a fixed minimum degree and bipartite graphs\thanks{ $\dag$ Corresponding author. E-mail addresses: xush0928@163.com (S.\ Xu), kexxu1221@126.com
(K.\ Xu), ivan.damnjanovic@famnit.upr.si (I.\  Damnjanovi\'{c}).}}\author{{Shaohan Xu $^{a}$, Kexiang Xu $^{a,\dag}$ and Ivan Damnjanovi\'{c} $^{b,c}$}\\\\{\small $^{a}$ School of Mathematics, Nanjing University of Aeronautics and Astronautics,}\\
{\small Nanjing, Jiangsu 210016, PR China}\\
{\small $^{b}$ FAMNIT, University of Primorska, Koper, Slovenia}\\
{\small $^{c}$ Faculty of Electronic Engineering, University of Ni\v{s}, Ni\v{s}, Serbia}}

\maketitle
\begin{abstract}
The number of spanning trees in a graph $G$ is the total number of distinct spanning subgraphs of $G$ that are trees. In this paper we characterize the unique graph with a prescribed vertex (resp.\ edge) connectivity, minimum degree and order that attains the maximum number of spanning trees. Moreover, all the bipartite graphs are determined with a given vertex (resp.\ edge) connectivity and order maximizing the number of spanning trees.\\
\noindent{\bf Keywords:} spanning tree, vertex connectivity, edge connectivity, minimum degree, bipartite graph.\\
\noindent{{\bf  2020 Mathematics Subject Classification:}  05C05, 05C30, 05C35.}
\end{abstract}

\section{Introduction}
All graphs considered in this paper are finite, undirected and without loops. Let $G$ be a graph with vertex set $V(G)$ and edge set $E(G)$. For any vertex $v\in V(G)$, we denote by $N_{G}(v)$, or simply $N(v)$, the set of vertices adjacent to $v$ in $G$, where $d_{G}(v)=|N_{G}(v)|$ is the degree of $v$ in $G$. Let $N_{G}[v]=N_G(v)\cup \{v\}$ for $v\in V(G)$. Without causing any confusion, let $N_{S}(v)=N_{G}(v)\cap S$ and $d_{S}(v)=|N_{S}(v)|$ for any $v\in V(G)$ and  $S\subseteq V(G)$.  Let $\delta(G)$, or $\delta$ for short, be the {\it minimum degree} of vertices in $G$. For any $S\subseteq V(G)$, $G[S]$ is the subgraph of $G$ induced by $S$. For any $V'\subseteq V(G)$ and $E'\subseteq E(G)$, let $G-V'$ (resp.\ $G-E'$) be the subgraph of $G$ obtained by deleting all vertices in $V'$ and all incident edges (resp.\ all edges in $E'$) with them. The graph $G$ is {\it connected} if each pair of vertices is joined by a path in it. A subset $S\subseteq V(G)$ is called a {\it vertex cut}  of a connected graph $G$ if $G-S$ is disconnected or a single vertex. Similarly, a subset $U\subseteq E(G)$ is called an {\it edge cut} in a connected graph $G$ if $G-U$ is disconnected. The {\it vertex connectivity} $\kappa(G)$ of $G$ is the minimum cardinality of  all vertex cuts of $G$. The  {\it edge connectivity} $\kappa' (G)$ of $G$ can be defined analogously. For any positive integer $k$, we write the set $\{1, 2, \ldots, k\}$ as $[k]$.

The number of spanning trees of $G$, denoted by $\tau(G)$,  is an important, well-studied invariant in graph theory and has found applications in many other scientific areas. For instance, the spanning tree number $\tau(G)$ is a key parameter in Tutte polynomials  and it is exactly the order of the critical group \cite{C5,D3} (also known as the sandpile group in statistical physics or the Picard group in algebraic geometry) of the graph.  Moreover, the spanning tree number has a close relation to the partition function of the $q$-state Potts model in statistical mechanics \cite{W2}. The problem of counting spanning trees  is  a fundamental  and popular topic in mathematics, physics, computer science, and so on, and has been extensively studied  with some relevant results in
 \cite{C3,D1,R2,G1,G2,T3,Z1}.

Another interesting problem is to determine the maximum or minimum number of spanning trees from some special classes of graphs.  This problem  has some  applications in the areas of experimental design \cite{C1} and  network theory \cite{K1}. From the viewpoint of network theory, the graphs with a large number of spanning trees will have greater reliability \cite{C2}. Thus it is significant to determine the value (or the bound) of the number of spanning trees in the design of reliable probabilistic networks. On minimizing the number of spanning trees in graphs, Alon \cite{A1}, Kostochka \cite{K4} and Sereni and Yilma \cite{S3} studied  the asymptotics of the minimum number of spanning trees in  regular graphs; Bogdanowicz \cite{B6,B8} characterized  the graphs minimizing the number of spanning trees among all connected graphs with $n$ vertices and $m$ edges; Bogdanowicz \cite{B7,B9} considered the minimum number of spanning trees in chordal graphs with a fixed vertex connectivity; Gong {\it et al}. \cite{G8} obtained the graphs minimizing the number of spanning trees in bipartite graphs with a fixed cyclomatic number.  On maximizing the number of spanning trees in graphs, McKay \cite{M1} precisely determined the asymptotics of the maximum number of spanning trees in  $d$-regular  graphs; Kelmans \cite{K1}, Grimmett \cite{G4}, Grone and Merris \cite{G5} and  Das \cite{D2} established several upper bounds on $\tau(G)$ for various graphs with given parameters; Li {\it et al}. \cite{L1} found the graphs maximizing the number of spanning trees among the graphs with a given connectivity and chromatic number, respectively;  Tapp \cite{T2} provided  the upper bound on the number of spanning trees for grid graphs. Moreover,  some relevant results are in \cite{O1,P3} for multi-graphs.

Denote by $K_{n}$ the {\it complete graph} of order $n$. Let $\overline{G}$ be the {\it complement} of a graph $G$.  For two vertex-disjoint graphs $G$ and $H$, their {\it union} is denoted by $G \cup H$ and their {\it join} $G\vee H$ is  the graph  obtained from $G\cup H$ by adding all possible edges from $V(G)$ to $V(H)$.   For any $r, \delta, n \in \mathbb{N}$  with $\delta \ge r$ and $n \ge \delta + 1$, let $\mathbb{V}_{n,\delta}^{r}$ (resp.\ $\mathbb{E}_{n,\delta}^{r}$) be the set of graphs of order $n$ with vertex (resp.\ edge) connectivity $r$ and minimum degree $\delta$, and  $\mathbb{BV}_n^r$ (resp.\ $\mathbb{BE}_n^r$) be the set of bipartite graphs of order $n$ with vertex (resp.\ edge) connectivity $r$.

In this paper we firstly prove that in $\mathbb{V}_{n,\delta}^{r}$ the maximum number of spanning trees is obtained uniquely at $K_{r}\vee(K_{\delta-r+1}\cup K_{n-\delta-1})$. Furthermore, we determine that $M_{n-\delta-1,\delta+1}^{r}$ uniquely attains the maximum number of spanning trees in $\mathbb{E}_{n,\delta}^{r}$, provided $\delta > r$, where $M_{n-\delta-1,\delta+1}^{r}$ is the graph obtained by connecting $r$ disjoint edges between two vertex-disjoint complete graphs $K_{n-\delta-1}$ and $K_{\delta+1}$. Moreover, we determine  all the graphs  that maximize the number of spanning trees in  $\mathbb{BV}_n^r$ and  $\mathbb{BE}_n^r$, respectively.

\section{Preliminaries}
For any two distinct vertices $v, w \in V(G)$ in $G$, the graph $shift(G,v,w)$ is obtained from $G$ by deleting $vx$ and adding $wx$ for all $x\in N(v)\setminus N[w]$. The above
transformation is called the {\it shift transformation} \cite{B2,S1}.  In particular, we define the {\it partial shift transformation} as follows. The graph $G(v\xrightarrow{x~} w)$ is  obtained from $G$ by deleting the edge $vx$ and adding the edge $wx$ for a vertex $x\in N(v)\setminus N[w]$.  The following result presents the effect of the partial shift transformation of graphs on the number of spanning trees.

\begin{lemma}\label{xx1}
Let $G$ be a connected graph with  two distinct vertices  $v,w\in V(G)$ such that  $N_{G}(w)\subseteq N_{G}[v]$. Suppose that two distinct vertices $v_{1}, v_{2}\in N_{G}(v)\setminus N_{G}[w]$ are joined by a path $P^*$ such that  $V(P^*)\cap(N_{G}[w]\cup \{v\})=\emptyset$. Then $\tau(G(v\xrightarrow{v_{1}}w))> \tau (G).$
\end{lemma}
\begin{proof}
Let $N_{G}(w)\setminus \{v\}=\{w_{1},w_{2}, \ldots,w_{s}\}$ and $G'=G(v\xrightarrow{v_{1}}w)$. Let $\mathcal{T}_G$ and $\mathcal{T}_{G'}$ be the sets of spanning trees in $G$ and $G'$, respectively.  We will show that there is an injection from $\mathcal{T}_G$ to $\mathcal{T}_{G'}$. It suffices to prove that there is a corresponding  $T_{G'}\in\mathcal{T}_{G'}$  for any $T_{G}\in \mathcal{T}_G$.

If $T_{G}$ does not contain the edge $vv_1$, then the above result holds trivially with $T_{G'}=T_{G}$. If  $vv_1\in E(T_{G})$, we assume that $F=T_{G}(v\xrightarrow{v_{1}} w)$. If $F$ is a tree, then  $F\in\mathcal{T}_{G'}$. Then we take $T_{G'}=F$, which also means that  $T_G=T_{G'}(w\xrightarrow{v_{1}} v)$. If $F$ is not a tree,  from the construction of $F=T_{G}(v\xrightarrow{v_{1}} w)$, $F$ is a  disconnected subgraph with exactly one cycle. Then there exists some vertex $w_{i}\in N_{G}(w)\setminus \{v\}$ with $1\leq i\leq s$ contained in one cycle $C_{1}=v_{1}ww_{i}\cdots v_{1}$ of $F$. Note that $w_{i}\in N_{G}(v)$ and the edge $vw_{i}$ cannot be included in $F$  since vertices $v,w$ must be in two  distinct  components of $F$. Thus, we replace $ww_{i}$ with $vw_{i}$ in $F$ to get $T_{G'}$. Furthermore, $T_{G}$  can be also obtained from $T_{G'}$ by running the transformation $T_{G'}(w\xrightarrow{v_{1}} v)$ to make a disconnected graph $W$ with exactly one cycle $C_{2}=v_{1}vw_{i}\cdots v_{1}$ for some $w_{i}$ with $1\leq i\leq s$, then replacing $vw_{i}$ by $ww_{i}$ in $W$ (since  $ww_{i}\not\in E(W)$).  Combining the above two cases,  we can get a $T_{G'}\in \mathcal{T}_{G'}$  corresponding to each $T_G\in\mathcal{T}_G$  for $G'=G(v\xrightarrow{v_{1}}w)$.

Since $P^*$ contains two terminal vertices $v_1$ and $v_2$  in $G'$,  there is a path $P=wP^*v$ of $G'$ which contains no vertices in $N_{G}(w)$. Assume  that $T_{G'}$  contains $P$. Then $T_{G'}(w\xrightarrow{v_{1}}u)$ produces a disconnected graph with exactly one cycle $C_{3}=vv_{1}\cdots v_{2}v$  such that  $w_{i}\notin V(C_3)$ for any $1\leq i\leq s$. Then $C_{3}\neq C_{2}$, which implies that $T_{G'}$ does not have a corresponding $T_G$ by the previously defined rules. Hence, $|\mathcal{T}_{G'}|>|\mathcal{T}_{G}|$, completing the proof.
{\hfill}
\end{proof}

\begin{lemma}\label{x3}
Let $G$ be a non-complete connected graph with $e\in E(\overline{G})$. Then we have $\tau(G) < \tau(G+e)$, where $G+e$ denotes the graph obtained from $G$ by adding the edge $e$.
\end{lemma}

A weighted graph is a graph $G$ together with a weight function $\omega:E(G)\rightarrow\mathbb{R}$, which is equivalent to an electrical network in which each edge $e$ has a resistor with conductance $\omega(e)$. The word ``weighted" is omitted when $\omega(e)=1$ for each $e\in E(G)$. For a weighted graph $G$, we also use  $\tau(G)$ to denote  the sum of weights of spanning trees of $G$, where the weight of a spanning tree $T$ in $G$ is the product of weights of edges in $T$. That is,
\begin{displaymath}
\tau(G)=\sum_{T\in \mathcal{T}_G}\prod_{e\in E(T)}\omega(e),
\end{displaymath}
where $\mathcal{T}_G$ is the set of spanning trees of $G$. If $\omega(e)=1$ for each edge $e\in E(G)$, then $\tau(G)=|\mathcal{T}_G|$.
We list the effect of some well-known transformations in electrical networks on the number of spanning trees as follows.  We refer the readers to \cite{T1} for details.

\vspace*{0.20cm}
\indent{\bf Parallel edges merging.}  If $G'$ is obtained from a graph $G$ by merging two parallel edges with weights $a$ and $b$ in $G$  into a single edge with weight $a + b$, then $\tau(G)=\tau(G')$.

\vspace*{0.20cm}
\indent{\bf Serial edges merging.}  If $G'$ is obtained from a graph $G$ by merging two serial edges with weights $a$ and $b$ in $G$  into a single edge with weight $\frac{ab}{a+b}$, then $\tau(G)=(a+b)\tau(G')$.

\vspace*{0.20cm}
\indent{\bf Mesh-star transformation.}  If $G'$ is obtained from a graph $G$ by replacing a complete subgraph $K_{s}$ of $G$ with weight $1$ on every edge by a weight star subgraph $K_{1,s}^{\omega}$ with weight $s$  on every edge, then $\tau(G)=\frac{1}{s^{2}}\tau(G')$.

\vspace*{0.20cm}
Let $G$ be a graph with vertex set $V(G)=\{v_{1},v_{2}, \ldots, v_{n}\}$. The {\it generalized join graph} $G[H_{1},H_{2}, \ldots, H_{n}]$ is obtained from $G$ by replacing each vertex $v_{i}$ with a graph $H_{i}$ and joining each vertex in $H_{i}$ with each vertex in $H_{j}$ provided $v_{i}v_{j}\in E(G)$. For a graph $H$ on $n$ vertices, let $\mu_1(H) \geq \cdots \geq \mu_{n-1}(H) \geq \mu_{n}(H)=0$ be the
eigenvalues of the Laplacian matrix $L_{H}$ of $H$.  Zhou and Bu  \cite{Z1} gave a formula for counting spanning trees in $G[H_{1},H_{2}, \ldots, H_{n}]$.

\begin{theorem}[\hspace{1sp}{\cite{Z1}}]\label{x111}
Let $G_{0}=G[H_{1},H_{2}, \ldots, H_{n}]$  be a generalized join graph of a connected graph $G$ with $n_{i}=|V(H_{i})|$ for $1\leq i\leq n$. Then
\begin{displaymath}
\tau(G_0)=\prod_{v_{i}\in V(G)}\frac{\prod_{k=1}^{n_{i}-1}\left(\mu_{k}(H_{i})+\sum_{v_{j}\in N_{G}(v_{i})}n_{j}\right)}{n_{i}}\sum_{T\in \mathcal{T}_G}\prod_{v_{i}v_{j}\in E(T)}n_{i}n_{j}.
\end{displaymath}
\end{theorem}

\begin{remark}
In the original version of Theorem \ref{x111}, that is, Theorem $6.22$ of \cite{Z1}, the graphs $H_1,H_2,\ldots, H_n$ are required to be connected. However, through the proof of  Theorem $6.22$ in \cite{Z1}, it can be found that the connectivity condition of $H_1,H_2,\ldots, H_n$ is not needed.  Therefore, the result holds in Theorem \ref{x111} (just from the original proof).
\end{remark}

\begin{lemma}[\hspace{1sp}{\cite{L1}}]\label{x10}
Let $n=s+\sum_{i=1}^{t}n_{i}$. Then
\begin{displaymath}
\tau(K_{s}\vee(K_{n_{1}}\cup K_{n_{2}}\cup\cdots\cup K_{n_{t}}))=n^{s-1}s^{t-1}\prod_{i=1}^{t}(s+n_{i})^{n_{i}-1}.
\end{displaymath}
\end{lemma}

\begin{lemma}\label{x7}
For $s\geq0$ and $x\geq2$, the function
$f(x)=\frac{(s+x+1)^{x}}{(s+x)^{x-1}}$
is strictly increasing in $x$.
\end{lemma}
\begin{proof}
Since $s+x+1> s+x\geq 2$, we let
$h(x)=\ln f(x)=x\ln(s+x+1)-(x-1)\ln(s+x)$.
After taking the first derivative of $h(x)$, we have
\begin{displaymath}
\begin{split}
h'(x)&=\frac{x}{s+x+1}-\frac{x-1}{s+x}+\ln(s+x+1)-\ln(s+x)\\
&=\frac{s+1}{(s+x+1)(s+x)}+\ln\frac{s+x+1}{s+x}>0.
\end{split}
 \end{displaymath}
This completes the proof.
{\hfill}
\end{proof}

\begin{lemma}\label{aux_lemma_1}
For any $\alpha, \beta, \gamma \in \mathbb{R}$ such that $\alpha, \beta > 0$ and $\alpha + \beta > \gamma$, we have
\[
    \alpha^{\beta - \gamma} \beta^{\alpha - \gamma} \le \left( \tfrac{\alpha + \beta}{2} \right)^{\alpha + \beta - 2 \gamma} .
\]
\end{lemma}
\begin{proof}
Let $s = \alpha + \beta$ and let $f \colon (0, s) \to \mathbb{R}$ be the function defined by
\[
    f(x) = (s - x - \gamma) \ln x + (x - \gamma) \ln (s - x) .
\]
A quick computation gives $f'(x) = \frac{(s - \gamma) (s - 2x)}{x(s - x)} + \ln \frac{s - x}{x}
$
for any $x \in (0, s)$. Then $f'(x) > 0$  for any $x \in (0, \tfrac{s}{2})$, while $f'(x) < 0$  for any $x \in (\tfrac{s}{2}, s)$. This means that $f(x)$ attains its maximum value uniquely in $\tfrac{s}{2}$, which implies $f(\alpha) \le f(\tfrac{s}{2})$. Therefore,
$(\beta - \gamma) \ln \alpha + (\alpha - \gamma) \ln \beta \le (\alpha + \beta - 2\gamma) \ln \tfrac{\alpha + \beta}{2}$, the result follows.
\end{proof}

\begin{lemma}\label{aux_lemma_2}
For any $\alpha, \beta, \gamma, \xi \in \mathbb{R}$ such that $\gamma, \xi > 0$ and $\alpha, \beta \ge \xi$, we have
\[
    \alpha^{\alpha - \gamma} \beta^{\beta - \gamma} \le \xi^{\xi - \gamma} (\alpha + \beta - \xi)^{\alpha + \beta - \xi- \gamma}.
\]
\end{lemma}
\begin{proof}
Let $s = \alpha + \beta$ and let $f \colon [\xi, s - \xi] \to \mathbb{R}$ be the function defined by
\[
    f(x) = (x - \gamma) \ln x + (s - x - \gamma) \ln (s - x).
\]
It is straightforward to compute that $f'(x) = \frac{\gamma (2x - s)}{x(s - x)} + \ln \frac{x}{s - x}$,
which means that $f'(x) < 0$ holds for any $x \in (\xi, \frac{s}{2})$ and $f'(x) > 0$ holds for any $x \in (\frac{s}{2}, s - \xi)$. Therefore, $f(x)$ attains its maximum value in at least one of the points $\xi$ and $s - \xi$. Since
$f(\xi) = f(s - \xi) = (\xi - \gamma) \ln \xi + (s - \xi- \gamma) \ln (s - \xi)$,
 then the result follows from $f(\alpha) \le f(\xi)$.
\end{proof}

Suppose that $A$ and $B$ are two disjoint subsets of $V(G)$ in a graph $G$. We let $[A,B]=\{uv\in E(G):u\in A, v\in B\}$.

\begin{lemma}[\hspace{1sp}{\cite{G3}}]\label{x4}
Let $G$ be a graph with minimum degree $\delta$ and $U$ be a non-empty proper subset of $V(G)$. If $|[U,V(G)\setminus U]|\leq \delta-1$, then $|U|\geq \delta+1$.
\end{lemma}
Denote by $\mathbb{V}_{n}^{r}$  (resp.\ $\mathbb{E}_{n}^{r}$) the set of connected graphs of order $n$ with vertex (resp.\ edge) connectivity $r$.  Li, Shiu and Chang \cite{L1} derived the following result.
\begin{theorem}[\hspace{1sp}{\cite{L1}}]\label{100}
Let $G\in \mathbb{V}_{n}^{r}$ (or $\mathbb{E}_{n}^{r}$)  with $n\geq r+1$.  Then
$\tau(G)\leq r n^{r-1}(n-1)^{n-r-2},$
where the equality holds if and only if $G\cong (K_{1}\cup K_{n-r-1})\vee K_{r}$.
\end{theorem}

\section{Graphs with given connectivity and minimum degree}
\subsection{Vertex connectivity and minimum degree}

For any $r \in \mathbb{N}$, the unique  graph of smallest order with vertex connectivity $r$ is $K_{r + 1}$. Suppose that $G\in \mathbb{V}_{n,\delta}^{r}$ and $G \not\cong K_{r + 1}$, so that $G$ contains a vertex cut $S$ of cardinality $r$ such that $G_{1}, G_{2}, \ldots, G_{t}$ with $t \ge 2$ are  all the connected components of $G - S$. Since $d_{G}(u)\geq \delta$ for each vertex $u\in V(G_i)$, we have $|S\cup V(G_i)|\geq \delta+1$ for $1\leq i\leq t$. Then
$ n = t \, |S \cup V(G_i)| - (t-1) \, |S| \ge 2(\delta + 1) - r$ since $t\geq 2$.
 As it turns out, $K_{r + 1}$ is the only graph with vertex connectivity $r$ when  $n<2(\delta + 1) - r$. Thus, we  assume that $n\geq 2(\delta+1)-r$ in  the following main result.

\begin{theorem}\label{l1}
Let $G\in \mathbb{V}_{n,\delta}^{r}$ with $n\geq 2(\delta+1)-r$. Then
\begin{displaymath}
\tau(G)\leq r n^{r-1}(\delta+1)^{\delta-r}(r+n-\delta-1)^{n-\delta-2},
\end{displaymath}
where the equality holds if and only if $G\cong K_{r}\vee(K_{\delta-r+1}\cup K_{n-\delta-1})$.
\end{theorem}
\begin{proof}
Let $G^{*}\in \mathbb{V}_{n,\delta}^{r}$ attain the maximum number of spanning trees with  a vertex cut $S$ of $G^{*}$ of cardinality $r$. Assume that  $G^{*}-S=G_{1}\cup G_{2}\cup\cdots \cup G_{t}$ with $t\geq 2$, where each $G_{i}$ ($1\leq i \leq t$) is a connected  component of $G^{*}-S$ and $v_{0}$ is a vertex of degree $\delta$ in $G^{*}$. We first claim that $t = 2$. Indeed, if $t\geq 3$, whether $v_{0} \in S$ or not (assume, without loss of generality, that $v_{0}\in V(G_{1})$), we can  add all the possible edges between $G_{i}$ and $G_{j}$ for $2\leq i < j\leq t$, and the resultant graph still belongs to  $\mathbb{V}_{n,\delta}^{r}$ with more spanning trees, contradicting the choice of $G^*$ by  Lemma \ref{x3}. Thus, $t = 2$ holds as desired.

Let $|V(G_{i})|=n_{i}$ with $i = 1, 2$. Since $\delta(G^*)=\delta\geq r$, from the structure of $G^*$, we have $n_{i}\geq \delta-r+1$ for  $i=1,2$. Note that $n_{1}+n_{2}=n-r$ and  we let $\widetilde{G}=K_{r}\vee (K_{n_{1}}\cup K_{n_{2}})$. Then $G^{*}$ is a spanning subgraph of $\widetilde{G}$ and $\tau(G^{*})\leq\tau(\widetilde{G})$ by Lemma \ref{x3}.

Without loss of generality, assume that $n_{1}\geq n_{2}\geq \delta-r+1$.  By Lemma \ref{x10}, we have
$\tau(\widetilde{G})=r n^{r-1}(r+n_{1})^{n_{1}-1}(r+n_{2})^{n_{2}-1}.$
Note that if $ n_{2}= \delta-r+1$, then  $\widetilde{G}= K_{r}\vee(K_{\delta-r+1}\cup K_{n-\delta-1})\in\mathbb{V}_{n,\delta}^{r}$. If $n_{1}\geq n_{2}\geq \delta-r+2$, then we can replace the pair $(n_{1},n_{2}) $ with $(n_{1}+1,n_{2}-1)$ satisfying $n_{1}+1\geq n_{2}-1\geq \delta-r+1$. Note that $n_2-1<n_1$. By Lemma \ref{x7}  with $s=r$, we have
\begin{displaymath}
\frac{(r+n_{2})^{n_{2}-1}} {(r+n_{2}-1)^{n_{2}-2}} < \frac{(r+n_{1}+1)^{n_{1}}}{(r+n_{1})^{n_{1}-1}}.
\end{displaymath}
Hence, $(r+n_{1})^{n_{1}-1}(r+n_{2})^{n_{2}-1}<(r+n_{1}+1)^{n_{1}}(r+n_{2}-1)^{n_{2}-2}$,  which implies that  $\tau(\widetilde{G})\leq\tau(K_{r}\vee(K_{\delta-r+1}\cup K_{n-\delta-1}))$, with equality holding if and only if $\widetilde{G}\cong K_{r}\vee(K_{\delta-r+1}\cup K_{n-\delta-1})$. Combining $\tau(G^{*})\leq\tau(\widetilde{G})$ and the maximality of $G^{*}$, we have $G^*\cong K_{r}\vee(K_{\delta-r+1}\cup K_{n-\delta-1})$. This completes the proof.
\end{proof}

For a given $r \in \mathbb{N}$ and $n \ge r + 2$, let $h(x) = r n^{r-1}(x+1)^{x-r}(r+n-x-1)^{n-x-2}$ be a function with $x \geq r$ and $n \ge 2(x + 1) - r$. It is routine to check that $h(x)$ is decreasing.  Since $\mathbb{V}_{n}^{r}=\bigcup_{\delta\ge r}\mathbb{V}_{n,\delta}^{r}$, Theorem \ref{100} can be as a corollary of Theorem \ref{l1}.

\subsection{Edge connectivity and minimum degree}

Let $n_{1},n_{2},q$ be positive integers such that $n_{1}\geq n_{2}\geq q $ and $n_{1}+n_{2}=n$. Let $M_{n_{1},n_{2}}^{q}$ be  the graph on $n$ vertices obtained by inserting $q$ mutually disjoint edges between the vertex-disjoint union $K_{n_1}\cup K_{n_2}$ of two  complete graphs $K_{n_{1}}$ and $ K_{n_{2}}$.

\begin{lemma}\label{x112} Let $n_{1},n_{2},q$ be positive integers such that $n_{1}\geq n_{2}\geq q$ and  $n_{1}+n_{2}=n$. Then
\[
    \tau(M_{n_{1},n_{2}}^{q})=qn_{1}^{n_{1}-q-1}n_{2}^{n_{2}-q-1}(n_{1}n_{2}+n)^{q-1}.
\]
\end{lemma}
\begin{proof}
Let $V(K_{n_{1}})=\{u_{1},u_{2}, \ldots,u_{n_{1}}\}$ and $V(K_{n_{2}})=\{v_{1},v_{2}, \ldots, v_{n_{2}}\}$ and $E(M_{n_{1},n_{2}}^{q})=E(K_{n_{1}})\cup E(K_{n_{2}})\cup \bigcup_{1\leq i\leq q}\{u_{i}v_{i}\}$. Using the Mesh-star transformation twice, we can  replace  $K_{n_{1}}$ and $K_{n_{2}}$ by $K_{1,n_{1}}^{\omega}$ and  $K_{1,n_{2}}^{\omega}$, respectively, and obtain a weighted graph $G'$.  Furthermore, in $G'$, the subgraph $K_{1,n_{1}}^{\omega}$ has the vertex set $\{u_{1},u_{2}, \ldots ,u_{n_{1}}\}\cup \{o_{1}\}$ and the edge set $\{o_{1}u_{i} : 1\leq i\leq n_{1}\}$ in which each edge has the weight $n_{1}$, the subgraph $K_{1,n_{2}}^{\omega}$ has the vertex set $\{v_{1},v_{2}, \ldots, v_{n_{2}}\}\cup \{o_{2}\}$ and the edge set $\{o_{2}v_{i} : 1\leq i\leq n_{2}\}$ in which each edge has the weight $n_2$, while all the edges of $G'$ in $\bigcup_{1\leq i\leq q}\{u_{i}v_{i}\}$ have the weight $1$. By the Parallel edges merging, Serial edges merging and the Mesh-star transformation, we have
\begin{displaymath}
\begin{split}
\tau(M_{n_{1},n_{2}}^{q})&=\frac{1}{n_{1}^{2}}\frac{1}{n_{2}^{2}}\tau(G')=\frac{1}{n_{1}^{2}}\frac{1}{n_{2}^{2}}n_{1}^{n_{1}-q}n_{2}^{n_{2}-q} (n_{1}n_{2}+n_{1}+n_{2})^{q}\frac{qn_{1}n_{2}}{n_{1}n_{2}+n_{1}+n_{2}}\\
&=qn_{1}^{n_{1}-q-1}n_{2}^{n_{2}-q-1}(n_{1}n_{2}+n)^{q-1}.\qedhere
\end{split}
\end{displaymath}
\end{proof}

From Theorem \ref{100}, we  obtain the following theorem for $\kappa'(G)=\delta(G)=r$.

\begin{theorem}
Let $G\in \mathbb{E}_{n, r}^{r}$ with $n \ge r + 1$. Then
$\tau(G)\leq r n^{r-1}(n-1)^{n-r-2},
$
where the equality holds if and only if $G\cong (K_{1}\cup K_{n-r-1})\vee K_{r}$.
\end{theorem}

For any graph $G\in \mathbb{E}_{n,\delta}^{r}$ with $\delta > r$, let  $U=[A,B]$ be an  edge cut of $G$ with $|U|=r$. From Lemma \ref{x4}, we have $\delta+1\leq |A|,|B|\leq n-\delta-1$. Thus, $n\ge 2\delta+2$. In this subsection, we determine the extremal graph which attains the maximum number of spanning trees among the graphs in $\mathbb{E}_{n,\delta}^{r}$ with $\delta > r$ and $n\ge 2\delta+2$.
\begin{theorem}
Let $G\in \mathbb{E}_{n,\delta}^{r}$ with $\delta > r$ and $n \ge 2\delta + 2$.  Then
\[
    \tau(G)\leq r(\delta+1)^{\delta-r}(n-\delta-1)^{n-\delta-r-2}[(\delta+1)(n-\delta-1)+n]^{r-1},
\]
where the equality holds if and only if $G\cong M_{n-\delta-1,\delta+1}^{r}$.
\end{theorem}
\begin{proof}
Let $G^{*}\in \mathbb{E}_{n, \delta}^{r}$ have the maximum number of spanning trees and let $U=[A,B]$ be an arbitrary edge cut of $G^{*}$ with $|U|=r$. By Lemma \ref{x4}, we have $\delta+1\leq |A|\leq n-\delta-1$ and $\delta+1\leq |B|\leq n-\delta-1$. Let $v_{0}$ be a vertex of degree $\delta$ in $G^{*}$. Without loss of generality, we suppose that $v_{0}\in A$. For $G^{*}$, let $A_{1}=N_{A}(v_{0})$ and $A_{2}=A\setminus (\{v_{0}\}\cup A_{1})$. Similarly, let  $B_{1}=N_{B}(v_{0})$ and  $B_{2}=B\setminus B_{1}$.  By Lemma \ref{x3}, $G^{*}[A\setminus\{v_0\}]$ and $G^{*}[B]$ are both cliques. Next, we prove that $G^{*}\cong M_{n-\delta-1,\delta+1}^{r}$.

Let $G_{1}$ be the graph obtained from $G^{*}$ by  adding all the possible edges between $v_{0}$ and $A_{2}$. By Lemma \ref{x3}, we have $\tau(G^*)\leq \tau(G_1)$. Suppose that $d_{B}(u)\geq 2$ for some $u\in A$ in  $G_{1}$ and  let $u_{1}, u_{2}\in N_{B}(u)$ in $G_1$ be two distinct vertices. Then $u_{1}u_{2}\in E(G^{*})$ since $G^{*}[B]$ is a clique.  Since $|[A,B]|=r<\delta$ and $|A\setminus\{u\}|\geq\delta$, we conclude that there exists a vertex $w\in A\setminus\{u\}$ that satisfies $d_{B}(w)=0$. We delete the edge $uu_{2}$ and add the edge $wu_{2}$ in $G_1$. Since $|A|\geq \delta +1>r$, we can keep doing this transformation until we reach $d_B(u) \le 1$ for every $u \in A$ in $G_1$, and we can do a similar transformation on the vertices in $B$. Let $G_{2}$ be the resultant graph constructed from $G_{1}$ by the previous transformations. By Lemma \ref{xx1}, we have $\tau(G_1)\leq \tau(G_2)$, with equality holding if and only if $d_{B}(u)\leq1$ for any $u\in A$ and $d_{A}(u)\leq1$  for any $u\in B$ in  $G_{1}$. By combining this with $\tau(G^{*})\leq\tau(G_{1})$, we have
\begin{equation}\label{g2_eq}
\tau(G^{*}) \leq \tau (G_{1}) \leq \tau (G_{2}).
\end{equation}

Let $|A|=a$ and $|B|=b$ be such that $a \ge b \ge \delta + 1$ and $a + b = n$. We observe that $G_2 \cong M_{a, b}^{r}$.  By Lemma \ref{x112}, we have $\tau(G_{2})=ra^{a-r-1}b^{b-r-1}(ab+n)^{r-1}$.

If $b = \delta+1$, we get $G_{2}\cong M_{n-\delta-1,\delta+1}^{r}\in \mathbb{E}_{n,\delta}^{r}$ and
\begin{displaymath}
\tau(M_{n-\delta-1,\delta+1}^{r})=r(\delta+1)^{\delta-r}(n-\delta-1)^{n-\delta-r-2}[(\delta+1)(n-\delta-1)+n]^{r-1}.
\end{displaymath}
Now we suppose that $a\geq b\geq \delta+2$. We define
$g(a,b,r)=ra^{a-r-1}b^{b-r-1}(ab+n)^{r-1}$,
where $a+b=n$ and $a\geq b\geq \delta+2$. We replace the pair $(a,b)$ by $(a+1,b-1)$, satisfying $(a+1)+(b-1)=n$ and claim that $g(a+1,b-1,r) > g(a,b,r)$. By Lemma \ref{x7}, we have
$f(b-2)< f(a-1)$ for $s=1$, that is,
$\frac{b^{b-2}} {(b-1)^{b-3}} < \frac{(a+1)^{a-1}}{a^{a-2}}.$
Then
\begin{displaymath}
\begin{split}
\frac{g(a+1,b-1,r)}{g(a,b,r)}
&=\frac{(a+1)^{a-r}(b-1)^{b-r-2}[(a+1)(b-1)+n]^{r-1}}{a^{a-r-1}b^{b-r-1}(ab+n)^{r-1}}\\
&=\frac{(a+1)^{a-1}(b-1)^{b-3}}{a^{a-2}b^{b-2}}\left[\frac{ab[(a+1)(b-1)+n]}{(a+1)(b-1)(ab+n)}\right]^{r-1}\\
&\geq \frac{(a+1)^{a-1}(b-1)^{b-3}}{a^{a-2}b^{b-2}}\\
&>1.
\end{split}
\end{displaymath}
Thus, $g(a+1,b-1,r)>g(a,b,r)$. This implies that
\begin{equation}\label{g2_xsh1}
\tau(G_2)=g(a,b,r)\leq g(a+1,b-1,r)\leq\cdots\leq g(n-\delta-1,\delta+1,r)=\tau(M_{n-\delta-1,\delta+1}^{r}),
\end{equation}
with equality holding if and only if $b=\delta+1$, that is, $G_2\cong M_{n-\delta-1,\delta+1}^{r}$. Recall that $M_{n-\delta-1,\delta+1}^{r}\in \mathbb{E}_{n,\delta}^{r}$. From the maximality of $G^{*}$  and in conjunction with   \eqref{g2_eq} and \eqref{g2_xsh1}, we have $G^*\cong M_{n-\delta-1,\delta+1}^{r}$.
\end{proof}

\section{Bipartite graphs with given connectivity}

In the present section we find all the graphs that attain the maximum number of spanning trees on $\mathbb{BV}_n^r$ and $\mathbb{BE}_n^r$, respectively. For any $r \in \mathbb{N}$, any bipartite graph with vertex connectivity $r$ cannot have below $2r$ vertices, while $K_{r, r}$ (resp.\ $K_{r, r + 1}$) is the unique bipartite graph with vertex connectivity $r$ and order $2r$ (resp. $2r + 1$). The following theorem fully solves the spanning tree maximization problem on $\mathbb{BV}_n^r$.

\begin{theorem}\label{bipv_th}
    Suppose that $G \in \mathbb{BV}_n^r$ with  $r \in \mathbb{N}$ and $n \ge 2r$. Then
    \[
        \tau(G) \le r \cdot \lfloor \tfrac{n + 1}{2} \rfloor^{r - 1} \cdot \lfloor \tfrac{n - 1}{2} \rfloor^{\lceil \frac{n - 1}{2} \rceil - r} \cdot \lceil \tfrac{n - 1}{2} \rceil^{\lfloor \frac{n - 3}{2} \rfloor},
    \]
    with the equality holding if and only if:
    \begin{enumerate}[label=\textbf{(\alph*)}]
        \item $G$ is a graph obtained by adding a new vertex and attaching it to $r$ vertices from the bipartition set of size $\lceil \frac{n - 1}{2} \rceil$ in $K_{\lfloor \frac{n - 1}{2} \rfloor, \lceil \frac{n - 1}{2} \rceil}$;
        \item $G$ can be another graph obtained by adding a new vertex and attaching it to a vertex from the bipartition set of size $\frac{n}{2} - 1$ in $K_{\frac{n}{2} - 1, \frac{n}{2}}$ for $r = 1$ and even $n \ge 4$.
    \end{enumerate}
\end{theorem}

The remainder of the section will mostly focus on proving Theorem \ref{bipv_th}.  In fact, these extremal graphs belong to a larger family that we will rely on to carry out the proof. We thus introduce the auxiliary B-graphs as follows.

\begin{definition}
    For  $a_1, a_2, \ldots, a_6 \in \mathbb{N}_0$, the B-graph $B(a_1, a_2, a_3, a_4, a_5, a_6)$ is the bipartite graph of order $\sum_{i = 1}^6 a_i$ obtained from $6$ mutually disjoint independent sets $I_1, I_2, \ldots, I_6$ of vertices  with $|I_i| = a_i$ for $i\in[6]$ by connecting any two vertices $u$ and $v$  with  $u \in I_1$ and $v \in I_4 \cup I_5$; or $u \in I_2$ and $v \in I_4 \cup I_5 \cup I_6$; or $u \in I_3$ and $v \in I_5 \cup I_6$.
\end{definition}

\begin{figure}[h]
    \centering
    \subfloat[]{
        \centering
        \begin{tikzpicture}
            \node[state, minimum size=0.75cm, thick] (1) at (0, 0) {$a_1$};
            \node[state, minimum size=0.75cm, thick] (2) at (2, 0) {$a_2$};
            \node[state, minimum size=0.75cm, thick] (3) at (4, 0) {$a_3$};
            \node[state, minimum size=0.75cm, thick] (4) at (0, -2) {$a_4$};
            \node[state, minimum size=0.75cm, thick] (5) at (2, -2) {$a_5$};
            \node[state, minimum size=0.75cm, thick] (6) at (4, -2) {$a_6$};

            \draw[thick] (1) to (4);
            \draw[thick] (1) to (5);
            \draw[thick] (2) to (4);
            \draw[thick] (2) to (5);
            \draw[thick] (2) to (6);
            \draw[thick] (3) to (5);
            \draw[thick] (3) to (6);
        \end{tikzpicture}
    }
    \hspace{1cm}
    \subfloat[]{
        \centering
        \begin{tikzpicture}
            \node[state, minimum size=0.75cm, thick] (1) at (0, 0) {$a_1$};
            \node[state, minimum size=0.75cm, thick] (2) at (2, -2) {$a_2$};
            \node[state, minimum size=0.75cm, thick] (3) at (4, 0) {$a_3$};
            \node[state, minimum size=0.75cm, thick] (4) at (0, -2) {$a_4$};
            \node[state, minimum size=0.75cm, thick] (5) at (2, 0) {$a_5$};
            \node[state, minimum size=0.75cm, thick] (6) at (4, -2) {$a_6$};

            \draw[thick] (1) to (4);
            \draw[thick] (1) to (5);
            \draw[thick] (2) to (4);
            \draw[thick] (2) to (5);
            \draw[thick] (2) to (6);
            \draw[thick] (3) to (5);
            \draw[thick] (3) to (6);
        \end{tikzpicture}
    }
    \caption{Two visual representations of the graph $B(a_1, a_2, \ldots, a_6)$ where each vertex corresponds to an independent set whose size matches the assigned label.}
    \label{b_fig}
\end{figure}
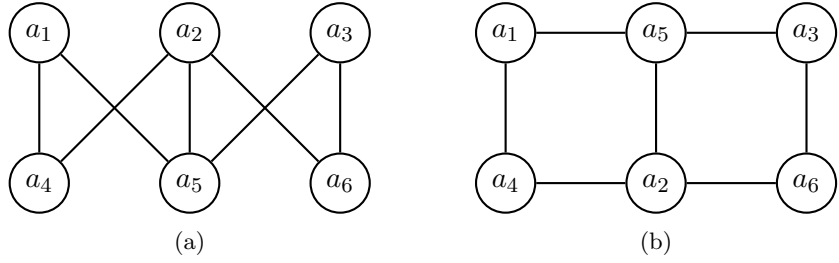

The B-graphs can be visually represented as shown in Figure \ref{b_fig}. Observe that the graph from item \textbf{(a)} of Theorem \ref{bipv_th} is isomorphic to $B(1, 0, \lfloor\tfrac{n - 1}{2}\rfloor, 0, r, \lceil \tfrac{n - 1}{2}\rceil - r)$, while the graph from item \textbf{(b)} is isomorphic to $B(1, 0, \frac{n}{2}, 0, 1, \frac{n}{2} - 2)$. We proceed with the next two lemmas.

\begin{lemma}\label{blemma2}
    Let $r \in \mathbb{N}$, $n \ge 2r + 2$ and  $a_3, a_6 \in \mathbb{N}_0$ with $a_3 + a_6 = n - r - 1$. Then
    \[
        \tau(B(1, 0, a_3, 0, r, a_6)) \le r \cdot \lfloor \tfrac{n + 1}{2} \rfloor^{r - 1} \cdot \lfloor \tfrac{n - 1}{2} \rfloor^{\lceil \frac{n - 1}{2} \rceil - r} \cdot \lceil \tfrac{n - 1}{2} \rceil^{\lfloor \frac{n - 3}{2} \rfloor},
    \]
    with the equality holding if and only if $(a_3, a_6) = (\lfloor \frac{n - 1}{2} \rfloor, \lceil \frac{n - 1}{2} \rceil - r)$; or $(a_3, a_6) = (\frac{n}{2}, \frac{n}{2} - 2)$  for $r = 1$ and even $n\ge 4$.
\end{lemma}
\begin{proof}
    Regardless of whether $a_6\neq0$, by applying Theorem \ref{x111} and performing a routine computation, we have
    \begin{equation}\label{aux_1}
        \tau(B(1, 0, a_3, 0, r, a_6)) = r a_3^{a_6} (a_3 + 1)^{r - 1} (a_6 + r)^{a_3 - 1}
    \end{equation}
    with $a_3 + a_6 = n - r - 1$ and $a_3 \ge 1$. In particular,
    \begin{equation}\label{aux_2}
        \tau(B(1, 0, \lfloor \tfrac{n - 1}{2} \rfloor, 0, r, \lceil \tfrac{n - 1}{2} \rceil - r)) = r \cdot \lfloor \tfrac{n + 1}{2} \rfloor^{r - 1} \cdot \lfloor \tfrac{n - 1}{2} \rfloor^{\lceil \frac{n - 1}{2} \rceil - r} \cdot \lceil \tfrac{n - 1}{2} \rceil^{\lfloor \frac{n - 3}{2} \rfloor}.
    \end{equation}

    Now, let
    $f(x) = r x^{n - r - 1 - x} (x + 1)^{r - 1} (n - 1 - x)^{x - 1}$ be a function on the interval $[1, n - r - 1]$.
    A routine computation gives
   \begin{displaymath}
   \begin{split}
        (\ln f)'(x) &= \frac{n - r - 1 - x}{x} + \frac{r - 1}{x + 1} - \frac{x - 1}{n - 1 - x} + \ln \frac{n - 1 - x}{x}\\
        &= \frac{-(2n - 4)x^2 + (n^2 - 5n + 5 + r)x + (n - 1)(n - r - 1)}{x(x + 1)(n - 1 - x)} + \ln \frac{n - 1 - x}{x}
        \end{split}
    \end{displaymath}
    for any $x \in (1, n - r - 1)$. Note that $x(x + 1)(n - 1 - x) > 0$ for $x \in (1, n - r - 1)$ and let
    \[
       P(x)= -(2n - 4)x^2 + (n^2 - 5n + 5 + r)x + (n - 1)(n - r - 1) .
    \]
    Since $-(2n - 4) < 0$ and $(n - 1)(n - r - 1) > 0$, Vieta's formulas imply that $P(x)$ has one positive and one negative root. Note that $P(\tfrac{n}{2} - 1) = \frac{n(n - r - 1)}{2} > 0$ and $P(\tfrac{n - 1}{2}) = - \frac{(n - 1)(r - 1)}{2} \le 0$.
    Then we have $(\ln f)'(x)>0$ for $x \in (1, \tfrac{n}{2} - 1)$ and $(\ln f)'(x)<0$ for $x \in (\tfrac{n - 1}{2}, n - r - 1)$.  Therefore,  $f(x)$ is strictly increasing on $[1, \tfrac{n}{2} - 1]$ and strictly decreasing on $[\tfrac{n - 1}{2}, n - r - 1]$.

     Now, let $G^* = B(1, 0, a_3^*, 0, r, a_6^*)$ be a graph from the set $\{ B(1, 0, a_3, 0, r, a_6) : a_3, a_6 \in \mathbb{N}_0, a_3 + a_6 = n - r - 1 \}$
    with the maximum number of spanning trees. Clearly, we have $a_3^* \ge 1$, otherwise, $G^*$ will be disconnected. We now divide the argument into two cases depending on  the parity of $n$.

    \vskip 0.25cm\noindent
    \textbf{Case 1:} $n$ is odd.

    Let $n = 2k + 1$. The monotonicity of $f(x)$ on $[1, k - \tfrac{1}{2}]$ and $[k, n - r - 1]$ implies that $a_3^* = k - 1$ or $a_3^* = k$. From \eqref{aux_1}, we get
    $\tau(B(1, 0, k - 1, 0, r, k - r + 1)) = r(k - 1)^{k - r + 1} k^{r - 1} (k + 1)^{k - 2}$ and
        $\tau(B(1, 0, k, 0, r, k - r)) = r k^{2k - r - 1} (k + 1)^{r - 1}$. Then
    \begin{displaymath}
   \begin{split}
        \frac{\tau(B(1, 0, k, 0, r, k - r))}{\tau(B(1, 0, k - 1, 0, r, k - r + 1))} = \frac{k^{2k - 2r}}{(k - 1)^{k - r + 1} (k + 1)^{k - r - 1}}.
    \end{split}
    \end{displaymath}
    Since $(k - 1)^{k - r + 1} (k + 1)^{k - r - 1} = (k^2 - 1)^{k - r - 1} (k - 1)^2 < (k^2)^{k - r - 1} \, k^2 = k^{2k - 2r}$, we obtain $\tau(B(1, 0, k, 0, r, k - r)) > \tau(B(1, 0, k - 1, 0, r, k - r + 1))$. Therefore, $a_3^* = k = \lfloor \frac{n - 1}{2} \rfloor$ and the result follows from \eqref{aux_2}.

   \vskip 0.25cm\noindent
    \textbf{Case 2:} $n$ is even.

    Let $n = 2k$. From the monotonicity of $f(x)$ on $[1, k - 1]$ and $[k - \frac{1}{2}, n - r - 1]$, we conclude that $a_3^* = k - 1$ or $a_3^* = k$. Formula \eqref{aux_1} gives
        $\tau(B(1, 0, k - 1, 0, r, k - r)) = r(k-1)^{k-r} k^{k + r - 3}$ and
        $\tau(B(1, 0, k, 0, r, k - r - 1)) = r(k - 1)^{k - 1} k^{k - r - 1} (k + 1)^{r - 1}$.
    Hence
    \begin{displaymath}
   \begin{split}
        \frac{\tau(B(1, 0, k - 1, 0, r, k - r))}{\tau(B(1, 0, k, 0, r, k - r - 1))} = \frac{k^{2r - 2}}{(k - 1)^{r - 1} (k + 1)^{r - 1}} = \left( \frac{k^2}{k^2 - 1} \right)^{r - 1} .
    \end{split}
    \end{displaymath}
    If $r \ge 2$, then $\tau(B(1, 0, k - 1, 0, r, k - r)) > \tau(B(1, 0, k, 0, r, k - r - 1))$, which means that $a_3^* = k - 1 = \lfloor \frac{n - 1}{2} \rfloor$. The result then follows from \eqref{aux_2}.
    On the other hand, if $r = 1$, then there are two non-isomorphic extremal graphs $B(1, 0, k - 1, 0, r, k - r)$ and $B(1, 0, k, 0, r, k - r - 1)$, i.e., $B(1, 0, \lfloor \frac{n - 1}{2} \rfloor, 0, r, \lceil \frac{n - 1}{2} \rceil - r)$ and $B(1, 0, \tfrac{n}{2}, 0, 1, \tfrac{n}{2} - 2)$, in accordance with the lemma statement.
\end{proof}

\begin{lemma}\label{blemma}
    Let $r \in \mathbb{N}$. Suppose that $a_1, a_2, \ldots, a_6 \in \mathbb{N}_0$ with $a_2 + a_5 = r$, $a_1, a_3 \ge a_5 + 1$, $a_4, a_6 \ge a_2 + 1$ and $n = \sum_{i = 1}^6 a_i$. Then
    \[
        \tau(B(a_1, a_2, \ldots, a_6)) < \tau(B(1, 0, \lfloor\tfrac{n - 1}{2}\rfloor, 0, r, \lceil \tfrac{n - 1}{2}\rceil - r)) .
    \]
\end{lemma}
\begin{proof}
    Regardless of whether $a_2 a_5 \neq 0$, by applying Theorem \ref{x111} and performing a routine computation, we get
    $\tau(B(a_1, a_2, \ldots, a_6)) = \tau_1(a_1, a_2, \ldots, a_6) \, \tau_2(a_1, a_2, \ldots, a_6)$,
    where
    \begin{equation}\label{tau_1_form}
    \begin{split}
        \tau_1(a_1, a_2, \ldots, a_6) =& \,\, (a_1 + a_2 + a_3)^{a_5 - 1} (a_4 + a_5 + a_6)^{a_2 - 1}\\
        &\cdot (a_1 + a_2)^{a_4 - 1} (a_2 + a_3)^{a_6 - 1} (a_4 + a_5)^{a_1 - 1}  (a_5 + a_6)^{a_3 - 1}
    \end{split}
    \end{equation}
    and
    \begin{equation}\label{uaux_8}
    \begin{split}
        \tau_2(a_1, a_2, \ldots, a_6) = & \,\, a_2^2 a_5^2 + a_2^2 a_5 (a_4 + a_6) + a_2 a_5^2 (a_1 + a_3) +  a_2^2 a_4 a_6 + a_5^2 a_1 a_3\\
        & + a_2 a_5 (a_1 + a_3)(a_4 + a_6) + a_2 a_4 a_6 (a_1 + a_3) + a_5 a_1 a_3 (a_4 + a_6) .
    \end{split}
    \end{equation}
    Recall that, in  the proof of Lemma \ref{blemma2}, we showed that the function $f(x) = r x^{n - r - 1 - x} (x + 1)^{r - 1} (n - 1 - x)^{x - 1}$
    is increasing on $[1, \tfrac{n}{2} - 1]$ and decreasing on $[\frac{n - 1}{2}, n - r - 1]$, while
  \[
        \tau(B(1, 0, \lfloor\tfrac{n - 1}{2}\rfloor, 0, r, \lceil \tfrac{n - 1}{2}\rceil - r)) = \max_{x \in [n - r - 1]} f(x).
    \]
    Therefore,  regardless of the parity of $n$, it suffices to prove that
   \begin{equation}\label{target_eq}
        \tau(B(a_1, a_2, \ldots, a_6)) < f(\tfrac{n}{2})=\frac{r n^{\frac{n}{2} - r - 1} (n + 2)^{r - 1} (n - 2)^{\frac{n}{2} - 1}}{2^{n - 3}}.
   \end{equation}

    First, if $r = 1$, without loss of generality, we assume that $(a_2, a_5) = (0, 1)$. In this case, we get $\tau(B(a_1, 0, a_3, a_4, 1, a_6)) = a_1^{a_4} (a_4 + 1)^{a_1 - 1} a_3^{a_6} (a_6 + 1)^{a_3 - 1} $ from Equations \eqref{tau_1_form} and \eqref{uaux_8}. By Lemmas \ref{aux_lemma_1} and \ref{aux_lemma_2}, we have
    \begin{equation}\label{xsh-xsh2}
    \begin{split}
        a_1^{a_4} (a_4 + 1)^{a_1 - 1} a_3^{a_6} (a_6 + 1)^{a_3 - 1} &\le \left( \tfrac{a_1 + a_4 + 1}{2} \right)^{a_1 + a_4 - 1} \left( \tfrac{a_3 + a_6 + 1}{2} \right)^{a_3 + a_6 - 1}\\
        &= \frac{(a_1 + a_4 + 1)^{a_1 + a_4 - 1} (a_3 + a_6 + 1)^{a_3 + a_6 - 1}}{2^{n - 3}}\\
        &\leq\frac{4^2 (a_1 + a_3 + a_4 + a_6 - 2)^{a_1 + a_3 + a_4 + a_6 - 4}}{2^{n - 3}}\\
        &=\frac{4^{2}(n - 3)^{n - 5}}{2^{n - 3}},
        \end{split}
    \end{equation}
    where the second inequality holds as  $a_1 + a_4 + 1 \ge 4$ and $a_3 + a_6 + 1 \ge 4$. Since  $n > n - 2 > n - 3 \ge 4$ for $n \ge 7$, we have $4^2 (n - 3)^{n - 5} < n^{\frac{n}{2} - 2} (n-2)^{\frac{n}{2} - 1}$. Applying this fact to \eqref{xsh-xsh2}, the inequality \eqref{target_eq} trivially follows when $r=1$.

    In the remainder of the proof, we assume that $r \ge 2$. We obtain appropriate upper bounds on $\tau_1(a_1, a_2, \ldots, a_6)$ and $\tau_2(a_1, a_2, \ldots, a_6)$ in the following claims, respectively.

\vskip 0.25cm
    \begin{claim}\label{xsh-xsh1}
        $\tau_1(a_1, a_2, \ldots, a_6) \le \frac{(r + 1)^r (n - r - 2)^{n - r - 5}}{2^{n - 2r - 4} (r + 2)} $.
\end{claim}
{\noindent\bf Proof of Claim \ref{xsh-xsh1}.} Let $g \colon [-a_5, a_2] \to \mathbb{R}$ be the auxiliary function defined by
    \begin{displaymath}
    \begin{split}
        g(x) =\,\, & (a_5 - 1 + x) \ln (a_1 + a_2 + a_3 + x) + (a_2 - 1 - x) \ln (a_4 + a_5 + a_6 - x)\\
        &+ (a_4 - 1 - x) \ln (a_1 + a_2) + (a_6 - 1 - x) \ln (a_2 + a_3)\\
        &+ (a_1 - 1 + x) \ln (a_4 + a_5) + (a_3 - 1 + x) \ln (a_5 + a_6) .
        \end{split}
    \end{displaymath}
    Note that $g(x)$ is well-defined since $a_1 + a_2 + a_3 + x > 0$ and $a_4 + a_5 + a_6 - x  > 0$ for any $x \in [-a_5, a_2]$. A straightforward computation leads to
    \begin{displaymath}
     \begin{split}
        g''(x)& = \frac{2(a_1 + a_2 + a_3) - a_5 + x + 1}{(a_1 + a_2 + a_3 + x)^2} + \frac{2(a_4 + a_5 + a_6) - a_2 - x + 1}{(a_4 + a_5 + a_6 - x)^2}\\
        &\geq \frac{2(a_1 + a_2 + a_3 - a_5) + 1}{(a_1 + a_2 + a_3 + x)^2} + \frac{2(a_4 + a_5 + a_6 - a_2) + 1}{(a_4 + a_5 + a_6 - x)^2}>0,
     \end{split}
    \end{displaymath}
    hence $g(x)$ with $x \in [-a_5, a_2]$ is a convex function. Therefore, $g(x)$ attains the maximum value in at least one of the points $-a_5$ and $a_2$. Now, observe that
    \begin{displaymath}
    \begin{split}
        g(0) &= \ln \tau_1(a_1, a_2, a_3, a_4, a_5, a_6),\\
        g(-a_5) &= \ln \tau_1(a_1 - a_5, r, a_3 - a_5, a_4 + a_5, 0, a_6 + a_5),\\
        g(a_2) &= \ln \tau_1(a_1 + a_2, 0, a_3 + a_2, a_4 - a_2, r, a_6 - a_2) .
        \end{split}
    \end{displaymath}
    With this in mind, while proving this claim, we can assume without loss of generality that $(a_2, a_5) = (0, r)$, $a_1, a_3 \ge r + 1$ and $a_4, a_6 \ge 1$.

    From \eqref{tau_1_form}, a quick computation and organization gives
    \begin{displaymath}
        \tau_1(a_1, 0, a_3, a_4, r, a_6) = \frac{a_1^{a_4 + r - 1} (a_4 + r)^{a_1 - 1} a_3^{a_6 + r - 1} (a_6 + r)^{a_3 - 1}}{(a_1 + a_3)(a_4 + a_6 + r)} \left( \frac{a_1 + a_3}{a_1 a_3} \right)^r .
    \end{displaymath}
     By Lemmas \ref{aux_lemma_1} and  \ref{aux_lemma_2}, we get
    \begin{equation}\label{uaux_1}
    \begin{split}
    \tau_1(a_1, 0, a_3, a_4, r, a_6)&\leq\frac{\left( \tfrac{a_1 + a_4 + r}{2} \right)^{a_1 + a_4 + r - 2}\left( \tfrac{a_3 + a_6 + r}{2} \right)^{a_3 + a_6 + r - 2}}{(a_1 + a_3)(a_4 + a_6 + r)} \left( \frac{1}{a_1}+\frac{1}{a_3} \right)^r \\
    &\leq\frac{\left( \tfrac{a_1 + a_4 + r}{2} \right)^{a_1 + a_4 + r - 2} \left( \tfrac{a_3 + a_6 + r}{2} \right)^{a_3 + a_6 + r - 2} \left( \tfrac{2}{r + 1} \right)^r}{(r + 2)(n - r - 2)} \\
    &=\frac{(a_1 + a_4 + r)^{a_1 + a_4 + r - 2} (a_3 + a_6 + r)^{a_3 + a_6 + r - 2}}{2^{n - 4} (r + 1)^r (r + 2)(n - r - 2)}\\
    &\leq \frac{(2r + 2)^{2r} (n - r - 2)^{n - r - 4}}{2^{n - 4} (r + 1)^r (r + 2)(n - r - 2)} \\
    &= \frac{(r + 1)^r (n - r - 2)^{n - r - 5}}{2^{n - 2r - 4} (r + 2)},
     \end{split}
   \end{equation}
    where the second inequality holds as $a_1,a_{3} \ge r + 1$ and $(a_1 + a_3)(a_4 + a_6 + r) \ge (r + 2)(a_1 + a_3 + a_4 + a_6 - 2) = (r + 2)(n - r - 2)$ from $a_1 + a_3\geq r+2$ and $a_4 + a_6 + r \ge r + 2$,  and the last inequality holds as $a_1 + a_4 + r \ge 2r + 2$ and $a_3 + a_6 + r \ge 2r + 2$. Then the claim follows from \eqref{uaux_1}.

\vskip 0.25cm
 \begin{claim}\label{xsh-xsh10}
        $\tau_2(a_1, a_2, \ldots, a_6) \le \frac{(n + r)^4}{256} $.
\end{claim}
{\noindent\bf Proof of Claim \ref{xsh-xsh10}.}
    Indeed, from \eqref{uaux_8}, a routine computation gives
    \begin{displaymath}
    \begin{split}
        \tau_2(a_1, a_2, \ldots, a_6) =& (a_2 + a')^2 (a_5 + a'')^2 - (a' a'')^2\\
        &- \frac{a_2 (a_1 + a_2 + a_3)(a_4 - a_6)^2 + a_5 (a_4 + a_5 + a_6)(a_1 - a_3)^2}{4} ,
      \end{split}
    \end{displaymath}
    where $a' = \frac{a_1 + a_3}{2}$ and $a'' = \frac{a_4 + a_6}{2}$. Therefore,
    \begin{equation}\label{uaux_xsh1}
    \begin{split}
    \tau_2(a_1, a_2, \ldots, a_6) \le (a_2 + a')^2 (a_5 + a'')^2\leq\left( \frac{(a_2 + a') + (a_5 + a'')}{2} \right)^4
 = \frac{(n + r)^4}{256}.
    \end{split}
   \end{equation}
  Then the claim follows from \eqref{uaux_xsh1}.

  By Claims  \ref{xsh-xsh1} and \ref{xsh-xsh10}, we have
        $\tau(B(a_1, a_2, \ldots, a_6)) \le \frac{(r + 1)^r (n - r - 2)^{n - r - 5} (n + r)^4}{2^{n - 2r + 4} (r + 2)} $.
  With \eqref{target_eq} in mind, we conclude that to finalize the proof, it suffices to show that
  \begin{equation}\label{uaux_6}
    2^{2r - 7} (r + 1)^r (n - r - 2)^{n - r - 5} (n + r)^4 < r(r + 2) n^{\frac{n}{2} - r - 1} (n + 2)^{r - 1} (n - 2)^{\frac{n}{2} - 1}.
    \end{equation}
    Note that $n \ge 3r + 4$ from the assumptions and let
    \begin{displaymath}
    \begin{split}
        h_r(x) =\, & \ln r + \ln(r + 2) + (\tfrac{x}{2} - r - 1) \ln x + (r - 1) \ln (x + 2) + (\tfrac{x}{2} - 1) \ln (x - 2)\\
        &- (2r - 7) \ln 2 - r \ln (r + 1) - (x - r - 5)\ln(x - r - 2) - 4 \ln (x + r).
    \end{split}
    \end{displaymath}
    be a  function  on $[3r + 4, +\infty)$.
    Then \eqref{uaux_6} is equivalent to $h_r(x) >0$ for $x\geq3r + 4$.

   Now, we show that $h_r(x)$ is an increasing function. A straightforward computation gives
    \begin{displaymath}
        h_r'(x) = - \frac{r + 1}{x} + \frac{r - 1}{x + 2} + \frac{3}{x - r - 2} - \frac{4}{x + r} + \tfrac{1}{2} \ln x + \tfrac{1}{2} \ln(x - 2) - \ln(x - r - 2),
    \end{displaymath}
    \vskip -0.8cm
    \begin{displaymath}
    \begin{split}
        h_r''(x) &= \frac{r + 1}{x^2} - \frac{r - 1}{(x + 2)^2} - \frac{3}{(x - r - 2)^2} + \frac{4}{(x + r)^2} + \frac{1}{2x} + \frac{1}{2(x - 2)} - \frac{1}{x - r - 2} .
        \end{split}
    \end{displaymath}
     For each $r \in \{2, 3, 4, 5\}$, it can be trivially verified by using any mathematical software that $h_r''(x) < 0$ holds for any $x > 3r + 4$. Now, suppose that $r \ge 6$. Observe that
    $\frac{1}{x} < \frac{1}{x - 2}$ and $\frac{1}{x + r} < \frac{1}{x - r - 2}$,
    hence
    \begin{equation}\label{uaux_10}
    \begin{split}
        h_r''(x) &< \frac{r + 1}{(x - 2)^2} - \frac{r - 1}{(x + 2)^2} - \frac{3}{(x - r - 2)^2} + \frac{4}{(x - r - 2)^2} + \frac{1}{x - 2} - \frac{1}{x - r - 2}\\
        &= \frac{2(x^2 + 4rx + 4)}{(x - 2)^2 (x + 2)^2} + \frac{1}{(x - r - 2)^2} - \frac{r}{(x - 2)(x - r - 2)} .
    \end{split}
    \end{equation}
    Observe that $\frac{x^2 + 4rx + 4}{x^2 + 4x + 4} < \frac{7}{3}$.
 Indeed, this is equivalent to $x^2 - (3r - 7) x + 4 > 0$,
    which holds on $[3r + 4, +\infty)$ with $r\geq6$.
   By combining this with \eqref{uaux_10}, we obtain
    \begin{displaymath}
    \begin{split}
        h''_r(x) &< \frac{14}{3(x - 2)^2} + \frac{1}{(x - r - 2)^2} - \frac{r}{(x - 2)(x - r - 2)}\\
        &= - \frac{(3r - 17)x^2 - (3r^2 - 16r - 68)x - (8r^2 + 44r + 68)}{3(x - 2)^2 (x - r - 2)^2}.
        \end{split}
    \end{displaymath}
     Let $q(x) = (3r - 17)x^2 - (3r^2 - 16r - 68)x - (8r^2 + 44r + 68)$ be a function with $x\ge 3r+4$.
    Since $\frac{3r^2 - 16r - 68}{2(3r - 17)} < \frac{r}{2} + 1$ holds for any $r \ge 6$, we conclude that $q(x)$ is increasing on $[3r + 4, +\infty)$, which implies that $q(x) \ge q(3r + 4) = 18r^3 - 53r^2 - 136r - 68 > 0$
    for any $x > 3r + 4$. Therefore, $h_r''(x) < 0$ for each $x > 3r + 4$, which implies that $h_r'(x)>0$ since $\lim_{x \to +\infty} h_r'(x) = 0$. Thus, $h_r(x)$ is an increasing function on $[3r + 4, +\infty)$.

    To finalize the proof, we need to show that $h_r(3r + 4) > 0$ for any $r \ge 2$.
     Let $z \colon [2, +\infty) \to \mathbb{R}$ be the auxiliary function defined as
    \begin{displaymath}
    \begin{split}
        z(x) =\, & \ln x + \ln (x + 2) + (\tfrac{x}{2} + 1) \ln (3x + 4) + (x - 1) \ln (3x + 6) + (\tfrac{3x}{2} + 1) \ln (3x + 2)\\
        &-(2x - 7) \ln 2 - x \ln (x + 1) - (2x - 1) \ln (2x + 2) - 4 \ln (4x + 4) .
        \end{split}
    \end{displaymath}
    Then $h_r(3r + 4) > 0$ for any $r \ge 2$ gets down to showing $z(x) > 0$ for $x\in [2, +\infty)$. We have
    \begin{displaymath}
    \begin{split}
        z'(x) =\,\, & \frac{1}{x} + \frac{x}{x + 2} + \frac{3x + 6}{6x + 8}
        + \tfrac{1}{2} \ln(3x + 4) + \ln(x + 2) + \tfrac{3}{2} \ln(3x + 2)- 3\ln (x + 1)\\
        & -\frac{3}{2}+\ln 3- 4 \ln 2,\\
        z''(x) =& - \frac{123x^5 + 604x^4 + 1204x^3 + 1232x^2 + 640x + 128}{x^2 (x + 1) (x + 2)^2 (3x + 2)(3x + 4)^2} .
    \end{split}
    \end{displaymath}
    Observe that $z''(x)< 0$ holds for any $x > 2$, while $\lim_{x \to +\infty} z'(x) = \ln \tfrac{27}{16} > 0$. Therefore, $z(x)$ is an increasing function, hence the result follows by verifying that $z(2) > 0$.
\end{proof}

We are now in a position to complete the proof of Theorem \ref{bipv_th}.

\begin{proof}[Proof of Theorem \ref{bipv_th}]
    The statement trivially holds if $n \in \{2r, 2r + 1\}$, so we assume that $n \ge 2r + 2$. Let $G^* \in \mathbb{BV}_n^r$ be a graph that maximizes the number of spanning trees with a vertex cut $S\subset V(G)$  of size $r$ such that $G_1, G_2, \ldots, G_t$  with $t \ge 2$ are all the connected components of $G^*-S$. Observe that $G^*[S \cup V(G_1)]$ is a complete bipartite graph. Indeed, otherwise, we could add an edge to $G^*$ without changing the graph order or vertex connectivity, thus yielding a contradiction by Lemma~\ref{x3}. By analogy, we conclude that $G^*[S \cup V(G_2) \cup \cdots \cup V(G_t)]$ is a complete bipartite graph. Therefore, $G^* \cong B(a_1, a_2, \ldots, a_6)$ for some $a_1, a_2, \ldots, a_6$ such that $a_1 + a_4 \ge 1$, $a_2 + a_5 = r$ and $a_3 + a_6 \ge 1$. We proceed by splitting the argument into two cases depending on whether $a_1 a_3 a_4 a_6 = 0$.

\vskip 0.25cm\noindent
    \textbf{Case 1:} $a_1 a_3 a_4 a_6 \neq 0$.

    In this case, we have $a_1, a_3, a_4, a_6 \ge 1$. Since removing the independent sets $I_1$ and $I_2$ from $B(a_1, a_2, \ldots, a_6)$ disconnects the graph, we conclude that $a_1 + a_2 \ge r$. Now, suppose that $a_1 + a_2 = r$. If we take a vertex from $I_1$ and connect it to all the vertices from $I_6$, the graph $G^*$ transforms to $B(a_1 - 1, a_2 + 1, a_3, a_4, a_5, a_6)$. Since $(a_1 - 1) + (a_2 + 1) = r$, we have $B(a_1 - 1, a_2 + 1, a_3, a_4, a_5, a_6) \in \mathbb{BV}_n^r$, yielding a contradiction to Lemma \ref{x3}. Therefore, $a_1 + a_2 \ge r + 1$, which implies $a_1 \ge a_5 + 1$. We can also analogously conclude that $a_3 \ge a_5 + 1$, $a_4 \ge a_2 + 1$ and $a_6 \ge a_2 + 1$. By virtue of Lemma \ref{blemma}, we again reach a contradiction.

  \vskip 0.25cm \noindent
    \textbf{Case 2:} $a_1 a_3 a_4 a_6 = 0$.

    Here we can assume, without loss of generality, that $a_4 = 0$ and trivially observe that $G^* \cong B(a_1, a_2, a_3, 0, a_5, a_6) \cong B(a_1, 0, a_2 + a_3, 0, a_5, a_6)$. Besides, if $a_5 < r$, then $a_1 \ge 1$ implies that $\kappa(G^*) < r$ as a clear contradiction. This leads us to $a_5 = r$ and $a_2 = 0$, hence $G^* \cong (a_1, 0, a_3, 0, r, a_6)$.

    Now, suppose that $a_6 \ge 1$. In this case, if we have $a_1 \ge 2$, then we can take a vertex from $I_1$ and connect it to all the vertices from $I_6$, thus transforming the graph $G^*$ to $B(a_1 - 1, 0, a_3 + 1, 0, r, a_6)$ which  preserves the vertex connectivity. By virtue of Lemma~\ref{x3}, we obtain a contradiction, implying  $a_1 = 1$. On the other hand, if we suppose that $a_6 = 0$, then $G^* \cong (a_1, 0, a_3, 0, r, 0) \cong (1, 0, a_1 + a_3 - 1, 0, r, 0)$. Either way, we may assume that $G^* \cong (1, 0, a_3, 0, r, a_6)$ for some $a_3, a_6 \in \mathbb{N}_0$ such that $a_3 + a_6 = n - r - 1$. The result now follows from Lemma \ref{blemma2}.
\end{proof}

 For any $r \in \mathbb{N}$, it is easy to see that any bipartite graph with edge connectivity $r$ cannot be of order below $2r$. Therefore, the next corollary of Theorem \ref{bipv_th} shows that the spanning tree maximization problem on $\mathbb{BE}_n^r$ has the same solution as on $\mathbb{BV}_n^r$.

\begin{corollary}\label{bipe_cor}
Suppose that $G \in \mathbb{BE}_n^r$ with $r \in \mathbb{N}$ and $n \ge 2r$. Then
\[
    \tau(G) \le r \cdot \lfloor \tfrac{n + 1}{2} \rfloor^{r - 1} \cdot \lfloor \tfrac{n - 1}{2} \rfloor^{\lceil \frac{n - 1}{2} \rceil - r} \cdot \lceil \tfrac{n - 1}{2} \rceil^{\lfloor \frac{n - 3}{2} \rfloor},
\]
with the equality holding if and only if:
\begin{enumerate}[label=\textbf{(\alph*)}]
    \item $G$ is a graph obtained by adding a new vertex and attaching it to $r$ vertices from the bipartition set of size $\lceil \frac{n - 1}{2} \rceil$ in $K_{\lfloor \frac{n - 1}{2} \rfloor, \lceil \frac{n - 1}{2} \rceil}$;
        \item $G$ can be another graph obtained by adding a new vertex and attaching it to a vertex from the bipartition set of size $\frac{n}{2} - 1$ in $K_{\frac{n}{2} - 1, \frac{n}{2}}$ for $r = 1$ and even $n \ge 4$.
\end{enumerate}
\end{corollary}
\begin{proof}
Observe that the graphs from items \textbf{(a)} and \textbf{(b)} of Theorem \ref{bipv_th} have the same vertex and edge connectivity, that is, they belong to $\mathbb{BE}_{n}^{r}$. Now, let $G^* \in \mathbb{BE}_n^r$ with the maximum number of spanning trees on $\mathbb{BE}_n^r$. Recall that $\kappa(G) \leq \kappa'(G) \leq \delta(G)$ holds for any graph $G$. Therefore, we have $G^*\in\mathbb{BV}_n^\ell$ for some $\ell \in [r]$. By Theorem~\ref{bipv_th}, we conclude that $\tau(G^*) \le \tau(\widetilde{G})$, where $\widetilde{G}$ arises from $K_{\lfloor \frac{n - 1}{2} \rfloor, \lceil \frac{n - 1}{2} \rceil}$ by adding a new vertex and attaching it to $\ell$ vertices from the bipartition set of size $\lceil \frac{n - 1}{2} \rceil$. If $\ell < r$, then we obtain a contradiction by Lemma \ref{x3} since $\widetilde{G}$ is a proper subgraph of the graph from item \textbf{(a)} of Theorem \ref{bipv_th}. Therefore, $\ell = r$, hence the result follows from Theorem~\ref{bipv_th}.
\end{proof}

\section*{Acknowledgements}
The work was supported by National Natural Science Foundation of China (Grant No.\ 12271251), Postgraduate Research \& Practice Innovation Program of Jiangsu Province, grant number KYCX25\_0625, the Ministry of Science, Technological Development and Innovation of the Republic of Serbia, grant number 451-03-137/2025-03/200102, and the Science Fund of the Republic of Serbia, grant \#6767, Lazy walk counts and spectral radius of threshold graphs --- LZWK.


\end{document}